\documentclass[5p]{elsarticle}
\usepackage{amssymb}
\usepackage{amsmath}
\usepackage{amsthm}
\usepackage{amssymb,color}
\usepackage{multirow}
\usepackage{algorithm}
\usepackage{algorithmic}
\usepackage{graphicx,subfigure,caption}
\usepackage{url} 

\def\IR{\mathbb{R}}
\def\IX{\mathcal{X}}

\newtheorem{theorem}{Theorem}[section]
\newtheorem{lemma}{Lemma}[section]

\newtheorem{assumption}{Assumption}[section]
\newtheorem{remark}{Remark}[section]

\journal{Operations Research Letters}

\begin{document}

\begin{frontmatter}



%\title{Better generalization via coupled tensor norm}
\title{Improving the generalization via coupled tensor norm regularization}
\author[1]{Ying Gao}
\ead{yinggao@buaa.edu.cn}
\author[1]{Yunfei Qu}
\ead{yunfei19@buaa.edu.cn}

\author[1]{Chunfeng Cui\corref{cor1}}
\ead{chunfengcui@buaa.edu.cn}
\author[1]{Deren Han}
\ead{handr@buaa.edu.cn}
\cortext[cor1]{Corresponding author}

\address[1]{
	School of Mathematical Sciences, Beihang University, 100191, Beijing, People’s Republic of China}
%LMIB of the Ministry of Education, School of Mathematical Sciences, Beihang University,
%Beijing, 100191, 

\begin{abstract}
% In this paper, we propose an effective method to improve the model generalization by the coupled tensor norm regularization. The novel regularization could enable the model output feature coupled with the data input to lie in a low-dimensional manifold, which helps us to reduce overfitting. We show that this regularization term is convex, differentiable, and gradient Lipschitz continuous for logistic regression, while nonconvex and nonsmooth for deep neural networks. %We demonstrate that the coupled tensor norm regularization could enforce the model output feature to lie in a low-dimensional manifold coupled with the data input, which helps us to reduce overfitting. 
% %Our proposed regularization is data-dependent and efficient especially for small data problems. 
% As a data-dependent regularization, this approach is efficient especially for small data problems. We leverage the first-order method to fit in the gradient descent framework and analyze their global 
% convergence. The results of numerical experiments on nine datasets demonstrate that our proposed method could improve the generalization performance of both multinomial logistic regression and deep neural network models. %on nine datasets.

In this paper, we propose a coupled tensor norm regularization that could enable the model output feature and the data input to lie in a low-dimensional manifold, which helps us to reduce overfitting. We show this regularization term is convex, differentiable, and gradient Lipschitz continuous for logistic regression, while nonconvex and nonsmooth for deep neural networks. We further analyze the convergence of the first-order method for solving this model. The numerical experiments demonstrate that our method is efficient.

\end{abstract}


\end{frontmatter}


\section{Introduction}
Modern machine learning models usually have numerous parameters, which will lead to overfitting and poor performance on unseen samples,   especially when the available training samples are few. 
Regularization functions are widely used to  constrain directly the parameters or internal structure of the model itself, thereby preventing overfitting to improve the generalization. Existing regularization techniques include  $\ell_1$ regularization \cite{shekar2020l1}, $\ell_2$ regularization \cite{cortes2012l2}, and Tikhonov regularization \cite{bishop1995training}. Further, for deep neural networks, regularization methods  such as %parameterized regularization \cite{sanyal2019stable}, 
weight decay \cite{krogh1991simple}, dropout \cite{srivastava2014dropout}, and batch normalization \cite{ioffe2015batch}  can also improve the generalization performance. Indeed, 
most regularization functions are  data-independent.

Besides, inherent data structure   can also bring about  data-dependent generalization approaches, such as %data distribution \cite{lyu2022improving}, 
data compression \cite{arora2018stronger}, tensor dropout  \cite{kolbeinsson2021tensor}, and tensor decomposition \cite{panagakis2021tensor}. 
Typically, these approaches only focus on the geometry of the input data. %  and do not encourage the network to  produce geometrically meaningful features.  
 Recently, LDMNet \cite{LDMNet} studied
the geometry of both the input data and the output features to encourage the network to learn geometrically meaningful features.  
However, LDMNet requires  solving a series of    variational subproblems. 

%\begin{figure}[t]
%	\centering
%	\subfigure[our proposed regularization]{
%		\label{Fig.sub.1}
%		\includegraphics[width=0.2\textwidth]{fig/warpAR10P_ours.png}}\qquad
%	\subfigure[$\ell_1$-norm regularization]{
%		\label{Fig.sub.2}
%		\includegraphics[width=0.2\textwidth]{fig/warpAR10P_l1.png}}\\
%	\subfigure[$\ell_2$-norm regularization]{
%		\label{Fig.sub.3}
%		\includegraphics[width=0.2\textwidth]{fig/warpAR10P_l2.png}}\qquad
%	\subfigure[Tikhonov regularization]{
%		\label{Fig.sub.4}
%		\includegraphics[width=0.2\textwidth]{fig/warpAR10P_Tik.png}}
%	\caption{The output features learned from multinomial logistic regression  with different regularizers for AR10P face dataset. Data are visualized in two dimensions using t-SNE.} 
%	\label{fig:motivation}
%\end{figure}
In this paper, motivated by the observation that the input data and the output features should sample a collection of low-dimensional manifolds  \cite{LDMNet}, we propose   the coupled tensor norm regularization. This regularization forces  the input tensor and the output feature matrix to lie  in a low-dimensional manifold explicitly, and  enables better model generalization consequently.
Figure~\ref{fig:motivation} gives an  example  that illustrates the features generated by our proposed coupled tensor norm regularization are better separated than the $\ell_1$-norm, $\ell_2$-norm, and Tikhonov regularization functions. 
\begin{figure}[t]
	\centering
	\subfigure[our proposed coupled tensor norm regularization]{
		\label{Fig.sub.1}
		\includegraphics[width=0.2\textwidth]{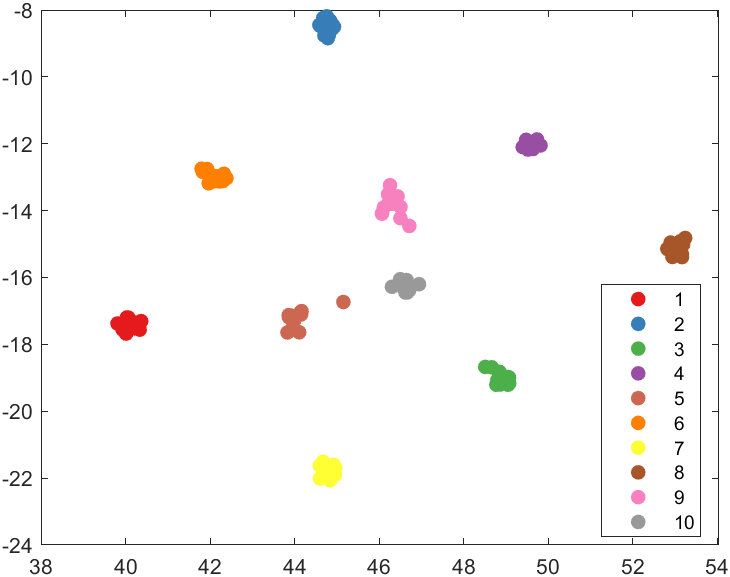}}\quad
	\subfigure[$\ell_1$-norm regularization]{
		\label{Fig.sub.2}
		\includegraphics[width=0.2\textwidth]{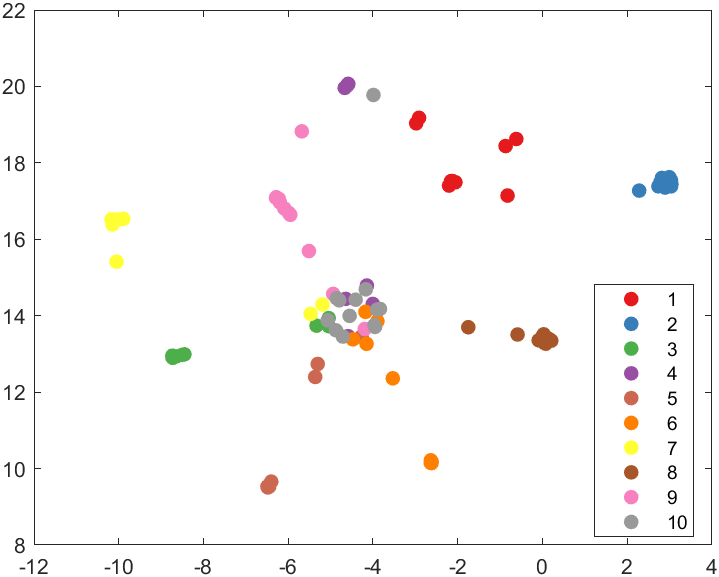}}\\
	\subfigure[$\ell_2$-norm regularization]{
		\label{Fig.sub.3}
		\includegraphics[width=0.2\textwidth]{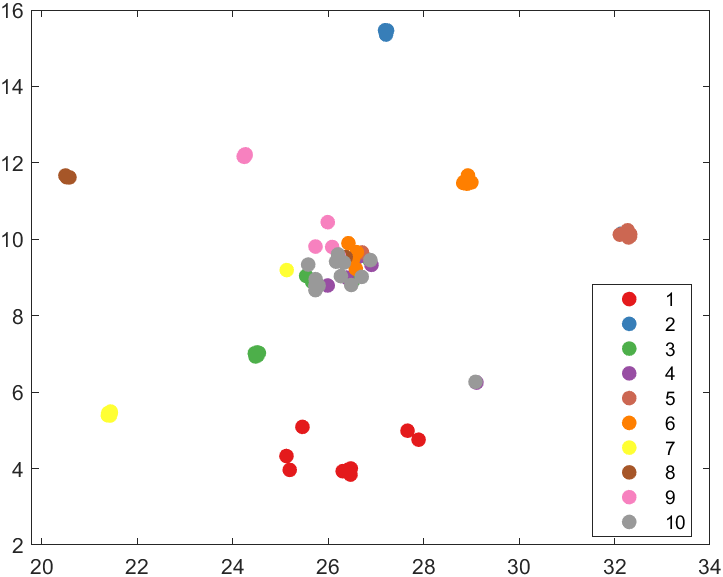}}\quad
	\subfigure[Tikhonov regularization]{
		\label{Fig.sub.4}
		\includegraphics[width=0.2\textwidth]{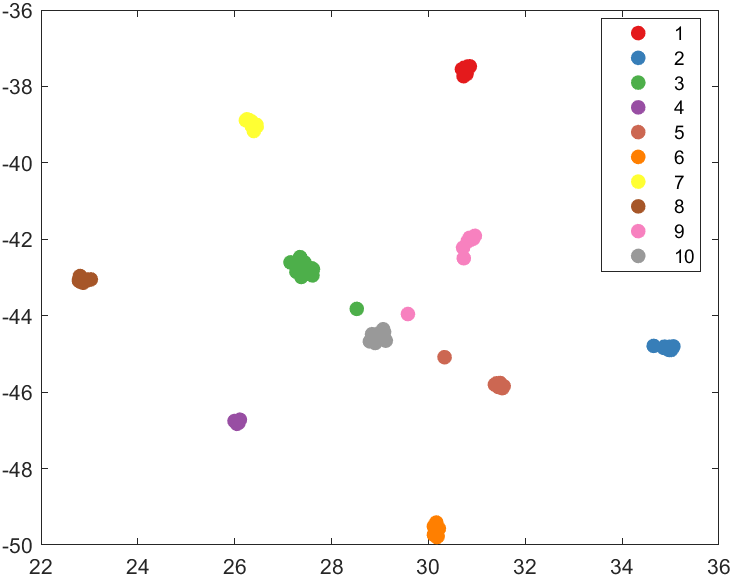}}
	\caption{The output features learned from multinomial logistic regression  with different regularizers for AR10P face dataset. Data are visualized in two dimensions using t-SNE.} 
	\label{fig:motivation}
\end{figure}

From the practical implementation point of view, there exist two difficulties in optimizing the loss function with the coupled tensor norm regularization. Firstly, this coupled norm is nonconvex and nonsmooth in general. Secondly, the  coupled tensor norm regularized model is nonseparable. Our main contributions are as follows.

\begin{itemize}
	\item We propose a coupled tensor norm regularization which enables the input tensor data and the output feature matrix to  lie   in a low dimensional manifold. Further, we present the convexity and smoothness analysis of this regularization  function.% to provide a basement for following models.
	 \item For multinomial logistic regression,  we show  the coupled tensor norm regularization is convex, differentiable, and its gradient is Lipschitz continuous. The gradient descent approach is adopted to solve the model and we 
	 analyze its global convergence.
	 
	 % Further, we adopt the gradient descent approach to solve this model and analysis its global convergence.  
	\item For deep neural networks,  we show the  loss function with the coupled tensor norm is nonconvex and  non-differentiable.  Further, we introduce an auxiliary variable  to derive the quadratic penalty formulation. Then we present an alternating minimization method  to  overcome the nonseparability. 
    We show this nonconvex and nonsmooth optimization problem is convergent  based on the K{\L} property.  
	  \item We conduct numerical experiments on a series of real datasets for both  multinomial logistic regression and deep neural networks. Compared with  the  $\ell_1$-norm, $\ell_2$-norm, and Tikhonov regularizations,  our proposed coupled norm regularization performs better in terms of testing accuracy, especially for small data problems.  
\end{itemize}
\subsection{Related work}
\textbf{Low-dimensional manifold.} As discussed in \cite{peyre2008image,peyre2009manifold},  
%the differential geometry has been explored for image processing and 
researchers proved that the patch set of many classes of images could come from a low-dimensional manifold. Afterwards, the low-dimensional manifold model (LDMM) \cite{osher2017low}  was raised  to compute  the dimension of the patch set based on the differential geometry. Then the dimension was  used as the regularization term for the image recovery problem. 
Recently, under an assumption that the concatenation of the input data and the output features should sample a collection of low-dimensional manifolds, LDMNet \cite{LDMNet} was proposed to regularize the deep neural network and outperformed some widely-used regularizers. 
All above models require  solving   the
variational subproblems to get a 
low-dimensional manifold. 
Specifically, the Euler-Lagrange equation caused by subproblem needs to be solved by using discretization and point integral method, which is complex relatively. 

\textbf{Tensor methods for generalization.} 
A recent line of generalization study focuses on tensor methods. For instance, \cite{Li2020understanding}  advanced the understanding of the relations between the network’s architecture and its generalizability from the
compression perspective. \cite{kolbeinsson2021tensor} proposed tensor dropout,  a randomization technique that can be applied to tensor factorizations, for improving generalization. Moreover, tensor methods \cite{cohen2016convolutional} can also be used to design efficient deep neural networks with stronger generalization. 
%because the weight convolutional layers are also represented as tensors.
Please see  \cite{panagakis2021tensor} for a review of  tensor methods  in computer vision and deep learning.  

\textbf{Coupled tensor norm.}
The joint factorization of coupled matrices and tensors (hereafter coupled tensors) has shown its power in   improving our understanding of the underlying structures in complex data sets  such as  social networks \cite{AcKoDu11}. Early studies on coupled factorization methods are usually nonconvex and the ranks of 
the coupled tensors need to be determined beforehand. 
To overcome the above bottlenecks, the coupled tensor norms \cite{signoretto2010nuclear,wimalawarne2014multitask,wimalawarne2018convex} were proposed for structured tensor computation, which are convex approaches for coupled tensor decomposition. 

% %As convex approaches for coupled tensor decomposition, coupled tensor norms mostly fall into two categories %\cite{tomioka2013convex}
% : the overlapped approach \cite{signoretto2010nuclear}  and the latent approach \cite{wimalawarne2014multitask}. 
% In this paper, we will adopt the overlapped approach because the latent approach requires the underlying tensor is almost full rank in all modes except one. 

\section{Preliminaries}
\textbf{Tensors.}  %A tensor is a generalization of vectors and matrices and is easily understood as a multidimensional array. 
%For a matrix $A\in\IR^{m\times n}$ with rank $r$, denote the singular value decomposition (SVD) of $A$ as $A=USV^\top$ and the maximal and minimal singular values of $A$ are denoted by $\sigma_{\min}(A)$ and $\sigma_{\max}(A)$, respectively. 
For an $N$-way tensor $\IX\in\IR^{I_1\times\cdots\times I_{N}}$, the mode-$n$ unfolding   matrix is $X_{(n)}\in\IR^{I_n\times J_n}$, where $J_n=\prod\limits_{i=1,i\neq n}^N I_i$.  
Further, the Tucker decomposition of $\IX$ is 
\begin{equation*}
	\IX = \mathcal{G}\times_1 U_1\times_2 U_2\times_3 \dots \times_N U_N,
\end{equation*}
or equivalently,
\begin{equation*}
	X_{(n)} = U_n G_{(n)}\left(U_N\otimes \dots \otimes U_{n+1}\otimes U_{n-1}\dots \otimes U_1\right)^\top,  
\end{equation*}
where  $U_n\in\IR^{I_n\times R_n}$ for $n=1,\dots,N$ are the factor matrices (which are usually orthogonal) and $\mathcal{G}\in\IR^{R_1\times\dots\times R_N}$ is the core tensor.  The multilinear rank of $\IX$ is defined as $(R_1,\cdots,R_N)$. Here, $R_n$ is also the rank of the mode-$n$ unfolding matrix of $\IX$. 
We refer to the reviewer paper \cite{kolda2009} for more details.

%\textbf{Notations.}
%
%the singular values vector as  $\sigma(A):=[\sigma_1(A),\cdots,\sigma_r(A)]$. 

%Given two matrices  $M\in\IR^{n_1\times n_2}$ and  $N\in\IR^{n_1\times n_2'}$, $[M,N]$ represents their concatenation on the column and $\|M,N\|_{*}$ denotes the joint matrix nuclear norm.

%\textbf{Tensors.}  
 %The relation between the index of the tensor element $a_{i_1,i_2,\cdots,i_N}$ and the index of an element $a_{i_n,j}$ of $A_{(n)}$ is 
%\begin{align*}
%	j=1+\sum\limits_{k=1,k\ne n}^N(i_k-1)J_k^{'},\; \text{where}\ \ J_k^{'}=\prod\limits_{m=1,m\ne n}^{k-1} I_m.
%\end{align*}

\textbf{Coupled tensor norm.}
Suppose the $n$-the order of tensor $\IX\in\IR^{I_1\times\cdots\times I_{N}}$ shares the same feature with the rows of   a matrix $A\in \IR^{I_n\times J}$.  
The coupling of $\IX$ and $A$ can provide more information on the features. 
 Please see Figure \ref{fig:coupled} for an example that a third order   tensor and  a matrix are coupled at mode-1. 
The  coupled low-rank decomposition of $\IX$ and $A$ 
seeks for the factors $U_i\in\IR^{I_i\times R_i}$ for $i=1,\dots,N$ and $V\in\IR^{J\times R_n}$ such that 
\begin{align*}%\label{eq:lowrank}
	\IX \approx \mathcal{G}\times_1 U_1\times_2 U_2\times_3 \dots \times_N U_N \text{ and } A \approx   U_nV^{\top}. 
\end{align*} 
Note that $U_n$ is shared between $\IX$ and $A$ with a  coupled rank   $R_n$. 
% The coupled tensor norm arises from convex approaches for low-rank decomposition of coupled tensor and matrix,  
% Overlapped approach is to unfold the coupled tensor and matrix into matrices and penalize these matrices simultaneously low-rank by nuclear norm. Corresponding 
However, the above equation needs to specify the rank   beforehand. 
To overcome this difficulty, \cite{wimalawarne2018efficient} proposed the coupled tensor norm of $[\IX,A]$ based on the overlapped approach, i.e., 
\begin{align}\label{eq:couplednorm}
	\|\IX,A\|_{\hbox{c}}:=\|X_{(n)},A\|_*+\sum\limits_{i=1,i\ne n}^N \|X_{(i)}\|_{*},
\end{align}
which arises from convex approaches for low-rank decomposition of coupled tensor and matrix.
Here,  the nuclear
norm function is  convex but nonsmooth. More specifically, for matrix $A$  with thin SVD decomposition $A=USV^\top$, its subgradient is 
\begin{align*}
	\partial \|A\|_{*}=\{UV^\top+Z\mid U^{\top}Z=0,ZV=0,\|Z\|\le 1\}.
\end{align*}  
%$\|A\|_*$ is the sum of singular values of $A$, which is a convex but nonsmooth function of $A$, and its subgradient is
%\begin{align*}
%	\partial \|A\|_{*}=\{UV^\top+Z\mid U^{\top}Z=0,ZV=0,\|Z\|\le 1\}.
%\end{align*}
Note that \eqref{eq:couplednorm}  does not need to specify the rank   beforehand and can  obtain the low-rank property automatically. 

\begin{figure}
	\centering
	\includegraphics[width=0.15\textwidth]{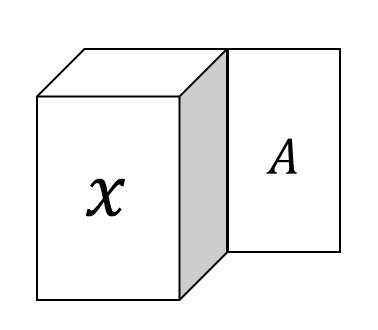}
	\caption{Illustration of the coupling between a tensor $\IX$ and  a matrix $A$ at the first mode.}
	\label{fig:coupled}
\end{figure}

\textbf{Variational analysis.} 
%At any point $\bar{x}\in C\subset \IR^n$, the set of all normal vectors and regular normal vectors are denoted by  $N_{C}(\bar{x})$  and $\hat{N}_{C}(\bar{x})$, respectively. A set $C\subset\IR^n$ is said to be Clarke regular at $\bar{x}\in C$ if it is locally closed at $\bar{x}$ and $N_{C}(\bar{x})=\hat{N}_{C}(\bar{x})$. 
Let $f:\IR^n\to \IR\cup\{+\infty\}$ be an extended function. $f$ is called regular at $\bar{x}$ if the set $\hbox{epi($f$)}$ is Clarke regular at $(\bar{x},f(\bar{x}))$. 
Here, $\hbox{epi}(f)=\{(x,r)\mid f(x)\le r  \}$ is the epigraph of $f$. Furthermore, if each $f_i\; (1\le i\le n)$ is regular at $\bar{x}$, then $f=\sum\limits_{i=1}^n f_i$ is regular at $\bar{x}$ and 
\begin{equation}\label{equ:subgrad_sum}
    \partial f(\bar{x})=\partial f_1(\bar{x})+\cdots+\partial f_n(\bar{x}).
\end{equation} 
The interested reader can refer to, e.g., \cite{rockafellar2009variational}, for more properties of Clarke regular and regularity.

 %$N_{C}(\bar{x})$ denotes the set of all normal vectors and $\hat{N}_{C}(\bar{x})$
%is the set of all regular normal vectors . 

\section{Method}
Consider the multi-classification with $n$ samples and $c$ disjoint classes. Let   $S=\{(\IX_i,y_i)\}_{i=1}^n$ denote independent and identically distributed training dataset, where $\IX_i\in\IR^{I_1\times\dots\times I_N}$ is the input data information and $y_i\in\IR^c$ is the corresponding label. Denote $m=\prod_{i=1}^N I_i$ as the number of input features. 
Here, $y_i=(y_i^{(1)},y_i^{(2)},\cdots,y_i^{(c)})^\top$ is the one-hot vector with  $y_i^{(k)}=1$ if  $\IX_i$ belongs to the $k$-th class  and $y_i^{(k')}=0$ otherwise. For convenience, let $\IX\in\IR^{n\times I_1\times\dots\times I_N}$ and $Y\in\IR^{n\times c}$ denote  the concatenation of $\IX_i$ and $y_i$ at the first mode, respectively. 
Also, denote   $f_{\theta}(\IX_i)\in\IR^c$ as the output feature of $\IX_i$.  
The loss function of the classification model is 
\begin{align}\label{eq:L}
	\min\limits_{\theta}\; L(\theta), \ 
	\text{where }  L(\theta)=\frac1n\sum_{i=1}^n l(f_{\theta}(\IX_i), y_i), 
\end{align}
where $l(\cdot)$ is the loss function, such as the cross entropy loss. 

Denote $f_{\theta}(\IX) \in \IR^{n\times c}$ as the concatenation of $f_{\theta}(\IX_i)$ at the first mode. As discussed in \cite{LDMNet}, the concatenation $\IX$ and $f_{\theta}(\IX)$ should sample a collection of low dimensional manifolds to reduce the risk of overfitting. 
%Please see Figure \ref{fig:coupled} for an example that a third order input tensor and  the output feature matrix are coupled at mode-1.  
This idea is due to the Gaussian mixture model: the tuples $\{(\IX_i,f_{\theta}(\IX_i))\}_{i=1}^n$ are generated by a mixture of low dimensional manifolds. 
Namely, suppose the coupled decomposition is
\begin{align*} 
	\IX \approx \mathcal{G}\times_1 U_1\times_2 U_2\times_3 \dots \times_{N+1} U_{N+1} \text{ and } f_{\theta}(\IX) \approx   U_1V^{\top},
\end{align*} 
where $U_1\in\mathbb{R}^{n\times R_1}$, $U_i\in\mathbb{R}^{I_{i-1}\times R_i}$ for $i=2,\dots,N+1$, and $V\in\mathbb{R}^{c\times R_1}$. 
Then the joint rank $R_1$ in the above formulation 
should be small. 
However, we do not know the joint rank in advance. 
Hence, the coupled tensor norm based on the overlapped approach  in \eqref{eq:couplednorm} is adopted to characterize low-rankness, and we have 
\begin{align}\label{ctn}
	\|\IX,f_{\theta}(\IX)\|_{c}=\|X_{(1)},f_{\theta}(\IX)\|_{*}+\sum\limits_{i=2}^{N+1}\|X_{(i)}\|_{*}.
\end{align}
Note that only the first term depends on $\theta$. By ignoring the terms independent of $\theta$, we propose the following regularization function, 
\begin{align}\label{eq:R}
	R(\theta)=\|X_{(1)},f_{\theta}(\IX)\|_*. 
\end{align}
Consequently, the classification model \eqref{eq:L} can be formulated as  the following general minimization problem, 
\begin{align}\label{eq:J+R}
	\min\limits_{\theta}\; L(\theta)+\lambda R(\theta), 
\end{align}
where $\lambda>0$ is the parameter.

 In the following theorem, we show the properties of the matrix concatenation function. 
\begin{theorem}\label{thm:XY_nuclear}
	Let  $X\in\mathbb R^{n\times m}$ and $\xi\in\mathbb R^{n\times c}$. Then the matrix row concatenation function $g(\xi)=\|X,\xi\|_{*}$ is not a norm. Further, it is a convex but nonsmooth function of $\xi$,  and   the subgradient of $g(\xi)$ is 
	\begin{equation*}%\label{eq:subg_gY}
		\partial g(\xi) = \{UV_2^{\top}+Z_2\,|\, U^{\top}Z=0, ZV=0, \|Z\|_2\le 1\},
	\end{equation*}
	where $U\Sigma V^{\top}$ is the thin SVD of $[\IX,\xi]$, $r$ is the rank of $[\IX,\xi]$,  $U\in\IR^{n\times r}$, $\Sigma\in\IR^{r\times r}$, $V\in\IR^{(m+c)\times r}$,  $V_2\in\IR^{c\times r}$ is the last $c$ rows of $V$, and $Z_2\in\IR^{n\times c}$ is the last $c$ columns of $Z\in\IR^{n\times (m+c)}$.  
\end{theorem}
\begin{proof}
	
	Note that when $\xi$ is equal  to zero, $g(\xi)$ may not equal zero. Hence, $g(\xi)$ is not a norm. 
	
	Further, by rewriting the concatenation of $X$ and $\xi$ as a linear function of $\xi$, 
	\begin{align}\label{equ}
		[X,\xi]=\xi A+[X,O],
	\end{align} 
	where $A=[O,I]$, $O$ is the $c$-by-$m$ zero matrix, and $I$ is the $c$-by-$c$ identity matrix, we have that $g(\xi)$ is the composition of  convex and linear functions. Hence,  $g(\xi)$ is convex consequently.  
	By the thin SVD decomposition $USV^{\top}$ of matrix $[X,\xi]$, the subdifferential of the nuclear norm   at $[X,\xi]$ is  
	\begin{align*}
		\partial \|X,\xi\|_{*}=\{UV^\top+Z\mid U^{\top}Z=0,ZV=0,\|Z\|\le 1\}.
	\end{align*}
	Through the chain rule and equation \eqref{equ}, we have 
	\begin{align}\label{partialg}
		\partial g(\xi)=\frac{\partial g(\xi)}{\partial [X,\xi]}\cdot\frac{\partial[X,\xi]}{\partial \xi}&=(UV^{\top}+Z)A^\top\nonumber\\
		&=UV^{\top}_2+Z_2,
	\end{align}
	where $V_2$ is the last $c$ rows of $V$ and  $Z_2$ is the last $c$ columns of $Z$. This completes the proof.
\end{proof}

Based on Theorem~\ref{thm:XY_nuclear}, we can analyze the properties of our proposed regularizer, which is the  composition of matrix nuclear norm, matrix concatenation, and classification model. 
\begin{theorem}\label{thm:3.2}
	Let  $R(\theta)$ be the regularizer defined by \eqref{eq:R}.  Then, 
	\begin{itemize}
		\item[(i)] if $f_{\theta}(\IX)$ is a linear function with respect to $\theta$, then $R(\theta)$ is   convex. Further, if $f_{\theta}(\IX)=X_{(1)}\theta$ or $f_{\theta}(\IX)=X_{(1)}\theta^{\top}$, then $R(\theta)$ is also smooth; 
		
		\item[(ii)]  if $f_{\theta}(\IX)$ is a general nonconvex and nonsmooth function with respect to $\theta$, then $R(\theta)$ is nonconvex and nonsmooth.
	\end{itemize}
\end{theorem}
\begin{proof}
If  $\xi=f_{\theta}(\IX)$ in Theorem~\ref{thm:XY_nuclear}  is a linear function with respect to $\theta$, then   $R(\theta)$ is the composition of convex and linear functions, and is convex consequently. On the other hand, when $f_{\theta}(\IX)$ is a general nonconvex and nonsmooth function with respect to $\theta$, $R(\theta)$ is nonconvex and nonsmooth. Furthermore, based on SVD decomposition $USV^{\top}$ of matrix $[X_{(1)},f_\theta(\IX)]$, we have
\begin{align*}
	X_{(1)}=USV_1^{\top},\;\; f_{\theta}(\IX)=USV_2^{\top},
\end{align*}
where $V_1,V_2$ are the front $m$ rows and the last $c$ rows of $V$, respectively.  
From \eqref{partialg} in the proof of Theorem~\ref{thm:XY_nuclear}, we have if $f_{\theta}(\IX)=X_{(1)}\theta$,
\begin{align*}
	\nabla R(\theta)=V_1SU^{\top}(UV_2^{\top}+Z_2)=V_1SV_2^{\top},
\end{align*}
and if $f_{\theta}(\IX)=X_{(1)}\theta^{\top}$, 
\begin{align} \label{XW^T}
	\nabla R(\theta)=(UV_2^{\top}+Z_2)^{\top}USV_1^{\top}=V_2SV_1^{\top}.
\end{align}
Hence, $R(\theta)$ is smooth in these two cases. The proof is completed. 
\end{proof}

\subsection{Multinomial logistic regression}
Multinomial logistic regression (MLR) is a  classical learning method for classification.
Let $x_i\in\IR^m$ be the vectorization of the input tensor $\IX_i$ and $X\in\IR^{n\times m}$ denote  all input features of $n$ training samples which is the unfolding matrix of   $\IX$ at the first mode. The probability   that $x_i$ belongs to the $k$-th class is modeled by the softmax function as follows,
\begin{align*}
	P(y_i^{(k)}=1\mid x_i,W):=\frac{\exp(w_k^\top x_i)}{\sum\limits_{j=1}^c \exp(w_j^\top x_i)},
\end{align*}
where $w_j\in\IR^m$ is the $j$-th weight vector, $W:=(w_1,w_2,\cdots\\, w_c)^\top\in\IR^{c\times m}$ is the weight matrix to be estimated from the training set. 
Furthermore, MLR   adopts the maximum likelihood estimation to estimate the weight matrix $W$. The corresponding loss function is 
\begin{align*}
	L(W)=-\frac1n\sum\limits_{i=1}^n\big(\sum\limits_{j=1}^c y_i^{(j)}w_j^\top x_i-\log\sum\limits_{j=1}^c \exp(w_j^\top x_i)\big),
\end{align*}
which is corresponding to equation \eqref{eq:L} with $\theta=\{W\}$, $f_\theta(X)$ being a linear formulation $XW^{\top}$, and $l(\cdot)$ being the cross entropy loss. 

The above   MLR model may  fail to generalize well on the testing dataset when the number of data $n$ is much smaller than the number of data features $m$ \cite{ZHANG2022148}. 
% which means most features are redundant.  
%In order to address this obstacle, the coupled nuclear norm regularization as already mentioned in \eqref{eq:R} can characterize the low-dimensional manifold of input training data and the output feature. 
%As  $f_\theta(X)$ in \eqref{eq:R} is equal to $XW^{\top}$ for MLR, t
To improve the generalization, the proposed model \eqref{eq:J+R} can be written as 
\begin{align}\label{MLR_R}
	\min\limits_{W\in\IR^{c\times m}}\;G(W):= L(W)+\lambda\|X,XW^\top\|_{*}.
\end{align}

It is well known that the nuclear norm function itself is nonsmooth. However, the coupled nuclear norm proposed in this context is differentiable with respect to $W$. More details are presented in the following lemma.

\begin{lemma}\label{diffR}
	Let regularization term $R(W)=\|X,XW^\top\|_{*}$ in model \eqref{MLR_R}, and $USV^\top$ be its SVD decomposition. Divide  $V$ as $[V_1,V_2]$ with $V_1\in\IR^{n\times m}$ and $V_2\in\IR^{n\times c}$.  Then we have
	\begin{itemize}
		\item[(i)]  $R(W)$ is differentiable with respect to $W$  and its gradient is $V_2SV_1^\top$; 
		\item[(ii)]  $\nabla R(W)$ is Lipschitz continuous if all singular values of matrix $XW^\top$ are nonzeros. 
	\end{itemize}
\end{lemma}
\begin{proof}
	The regularization $\|X,XW^{\top}\|_{*}$ in model \eqref{MLR_R} is a special case of \eqref{eq:R}. Hence, we can get  the gradient of $R(W)$ from \eqref{XW^T}, i.e., 
	%$\nabla R(W)=V_2SV_1^{\top}$. 
	\begin{align*}
		\nabla R(W)=V_2SV_1^{\top}.
	\end{align*}

	Furthermore, for any matrices $W,\hat W\in\IR^{c\times m}$, we   have 
	\begin{align*}
		\|\nabla R(W)-\nabla R(\hat{W})\|_{F}
		=\|V_2 S V_1^\top - \hat{V}_2^\top\hat{S}\hat{V}_1^\top\|_{F}=\|CB\|_{F},
	\end{align*}
	where $B=\begin{pmatrix}USV_1^\top\\ -\hat{U}\hat{S}\hat{V}_1^\top \end{pmatrix}$, $C=(V_2U^\top,\hat{V}_2^\top\hat{U}^\top)$. Owing to the fact that the largest singular value of $C$ satisfies $\sigma_{\max}(C)\le 1$, we have 
	\begin{align}\label{nablaR1}
		\|\nabla R(W)-\nabla R(\hat{W})\|_{F}\le\|B\|_{F}.
	\end{align}
	By denoting $E=\begin{pmatrix}
		V_2SU^\top,\hat{V}_2\hat{S}\hat{U}^\top
	\end{pmatrix}=\begin{pmatrix} WX^{\top}, \hat W X^{\top}\end{pmatrix}$, we have
	\begin{align}\label{nablaR2}
		&\|X^\top XW^\top -X^\top X\hat{W}^\top\|_{F} \nonumber \\
		=&\|V_2 S^2 V_1^\top -\hat{V}_2\hat{S}^2\hat{V}_1^\top\|_{F}\nonumber\\
		=&\|V_2SU^\top USV_1^\top - \hat{V}_2\hat{S}\hat{U}^\top \hat{U}\hat{S}V_{1}^\top\|_{F}\nonumber\\
		=&\|EB\|_{F} \nonumber\\
		\ge & \sigma_{\min}(E)\|B\|_F,
	\end{align}
    where $\sigma_{\min}(E)$ denotes   the smallest singular value of 
 $E$. Under the assumption that all singular values of $XW^{\top}$ is nonzero for all $W$ and the fact that $E$ is the concatenation of $ WX^{\top}$ and $\hat W X^{\top}$, we have $\sigma_{\min}(E)>0$. 
	Combining \eqref{nablaR1} and  \eqref{nablaR2}, we have 
	\begin{align*}
		&\|\nabla R(W)-\nabla R(\hat{W})\|_{F}\\
		\le& \|B\|_{F}\\
		\le & \frac{1}{\sigma_{\min}(E)} \|X^\top XW^\top -X^\top X\hat{W}^\top\|_{F} \\
		\le& \frac{\lambda_{\max}(X^\top X)}{\sigma_{\min}(E)}\|W-\hat{W}\|_{F}.
	\end{align*}
	Hence,  $\nabla R(W)$ is Lipschitz continuous and this  completes the proof. 
\end{proof}

The intrinsic convexity and differentiability of model \eqref{MLR_R} inspire us to utilize 
%gradient algorithms. Though the number of train samples is small, too many features also lead to heavy calculation burden. The current paper adopt 
the classical gradient descent algorithm with linesearch to solve it.

\textbf{Algorithm description.} 
The iterative process for updating $W^{k+1}$ is based  on following procedure, 
\begin{align}\label{GDL}
	W^{k+1}=W^k+\alpha_k\nabla G(W^k).
\end{align}
Here, $\alpha_k$ is obtained by a linesearch 
algorithm with guaranteed sufficient decrease which was introduced in \cite{more1994line}.

%, and the gradient of the objective function $G(W)$    consists of the following two parts, i.e., $\nabla G(W)=\nabla L(W)+\lambda \nabla R(W).$ 

For a given threshold $\varepsilon\ge0$, the  termination criterion is $\|\nabla G(W)\|_2\le\varepsilon$. The convergence analysis is illuminated in the next theorem.
\begin{theorem}
	Suppose all singular values of output matrix $XW^\top$ are nonzeros. For the convex differential 
	minimization problem \eqref{MLR_R}, each accumulation point of the iterative sequence $\{W^k\}_{k=0}^\infty$ generated by procedure \eqref{GDL} is a global minimizer.
\end{theorem}

\begin{proof}
The iterative procedure \eqref{GDL} is a gradient descent method with the stepsize $\alpha_k$ satisfying the Wolfe-Powell rules, and $\nabla G(W)$ in model \ref{MLR_R} is Lipschitz continuous based on Lemma \ref{diffR}. Hence, as introduced in \cite[Theorem 2.5.7]{sun2006optimization}, for the sequence $\{W^k\}_{k=0}^\infty$ generated by the  gradient descent method with Wolfe linesearch, either $\|\nabla G\\(W^k)\|_{2}=0$ for some $k$  or $\|\nabla G(W^k)\|_{2}\to 0$. It means that each accumulation point of the iterative sequence $\{W^k\}_{k=0}^\infty$ is a stationary point. Furthermore, the stationary point  is also a global minimizer owing to the convexity of model \eqref{MLR_R}. This  completes the proof.
\end{proof} 
\subsection{Deep neural networks (DNN)}
Let $\theta$ be the collection of network weights and bias. For every data $\IX_i$, the classical DNN learns a feature $f_{\theta}(\IX_i)\in\IR^c$ 
%, which represents the probability distribution of $x_i$ over the $c$ classes. 
by minimizing the empirical loss function on the training data as defined by \eqref{eq:L}. 
%However, when the training samples are scarce, statistical learning theories suggest that overfitting to the training data will occur \cite{vapnik1999overview}. 
To reduce the risk of overfitting of DNN \cite{vapnik1999overview}, we apply the coupled tensor norm regularization \eqref{eq:R} into the  loss function  and propose the regularized DNN model  as
\begin{equation}\label{eq:NN_L}
	\min_{\theta}~L(\theta)+\lambda \| X_{(1)}, f_{\theta}(\IX)\|_*,
\end{equation}
where $f_{\theta}(\IX)\in\IR^{n\times c}$ is   highly nonlinear and nonsmooth. Therefore, from Theorem \ref{thm:3.2}, the regularization term in model \eqref{eq:NN_L} is nonconvex, nonsmooth, and nonseparable.

A basic condition for solving DNN by the stochastic gradient descent (SGD) method is that the objective function is separable. 
To circumvent the nonseparability of \eqref{eq:NN_L}, we introduce an auxiliary variable $\xi$ into \eqref{eq:NN_L} as follows:
\begin{align*}\label{eq:NN_L2}
	\min_{\theta,\xi}~L(\theta)+\lambda \| X_{(1)}, \xi\|_*, 
	\text{  s.t.  } f_{\theta}(\IX) = \xi.
\end{align*}	
	Then we penalize the constraint into the loss function using the quadratic penalty method and  get the following unconstrained model: 
	\begin{equation}\label{eq:NN_L3}
		\min_{\theta,\xi}~\mathcal{L}(\theta,\xi):=L(\theta)+\lambda \| X_{(1)}, \xi\|_* +\frac{\mu}{2} \|f_{\theta}(\IX) - \xi\|^2_F,
	\end{equation}
	where $\mu>0$ is the penalty parameter. 
 Problem \eqref{eq:NN_L3} is solved by  alternating the directions of  $\theta$ and $\xi$. Specifically, given $(\theta^k,\xi^k)$, we implement the following sub-steps:
	\begin{itemize}
		\item Update $\theta^{k+1}$ with the fixed $\xi^k$:
		\begin{equation}\label{eq:NN_theta}
			\min_{\theta}~L(\theta)+\frac{\mu}{2} \|f_{\theta}(\IX) - \xi^k\|^2_F.
		\end{equation}
		\item Update $\xi^{k+1}$ with the fixed  $\theta^{k+1}$:
		\begin{equation}\label{eq:NN_xi}
			\min_{\xi}~\lambda \| X_{(1)}, \xi\|_* +\frac{\mu}{2} \|f_{\theta^{k+1}}(\IX) - \xi\|^2_F.
		\end{equation}
	\end{itemize}
	The $\theta$-subproblem \eqref{eq:NN_theta} is separable with respective to the samples $\IX_i$ and can be solved by SGD.     
	The $\xi$-subproblem  \eqref{eq:NN_xi}  is  strongly convex but   nonsmooth  by Theorem \ref{thm:XY_nuclear}. 
Based on the above analysis, we describe our  algorithm framework for solving \eqref{eq:NN_L3} in Algorithm \ref{alg:DNN}.
\begin{algorithm}
	\caption{The alternating minimization method for  \eqref{eq:NN_L3}}
	\label{alg:DNN}
	\begin{algorithmic}
		\REQUIRE Training data $\{(\IX_i,y_i)\}_{i=1}^n$, hyperparameters $\lambda$ and $\mu$, and a neural network with   initial weight  $\theta^0$.
		\ENSURE Trained network weights $\theta^*$.

		Let $k=0$. 
  %Randomly initialize the network weights $\theta^0$. 
  %The auxiliary variable 
  $\xi^0\in \IR^{n\times c}$ is initialized as zero.
		\WHILE{not converge}   
			\item 1. Update  $\theta^{k+1}$in\eqref{eq:NN_theta}: solve the nonconvex problem by SGD.
			\item 2. Update $\xi^{k+1}$ in \eqref{eq:NN_xi}: solve the convex problem by the subgradient method.
			\item 3. $k\leftarrow k+1$. 
		\ENDWHILE
		
		$\theta^* = \theta^k$.
	\end{algorithmic}
\end{algorithm}	

%The convergence of nonconvex and nonsmooth optimization problem \eqref{eq:NN_L3} is a challenging problem. 

Next, we show the global convergence of Algorithm \ref{alg:DNN} based on the K{\L} property and regularity under the following mild assumption.  
\begin{assumption}\label{ass:bound}
Assume the following conditions hold.
	\begin{itemize}
	    \item[(i)] The loss function $\mathcal{L}(\theta,\xi)$ defined by \eqref{eq:NN_L3}  is regular and satisfies the K{\L} property. 
  \item[(ii)] The subgradient of $f_\theta(\IX)$ is upper bounded, i.e., there exists a positive  constant $\rho$ such that for all 
	$g_f\in\partial f_{\theta}(\IX)$, we have $\|g_f\|\le\rho.$
\end{itemize}
\end{assumption}
\begin{theorem}\label{thm:DNN}
	Suppose Assumption \ref{ass:bound} holds.   Let the sequence $\{(\theta^k,\xi^k)\}_{k\ge0}$ be generated by Algorithm \ref{alg:DNN}. Then  $\{\xi^k\}_{k\ge0}$ has finite length and  converges to a point $\xi^*$   globally. Moreover, let $\theta^*$ be any limit point of the sequence $\{\theta^k\}_{k\ge0}$, then $(\theta^*,\xi^*)$ is a critical point of $\mathcal{L}$.
\end{theorem}
\begin{proof}
	We firstly show
	the sequence $\{(\theta^k,\xi^k)\}_{k\ge0}$ generated by Algorithm \ref{alg:DNN} satisfies the four conditions for the    bounded approximate gradient-like descent sequence in \cite[Definition 2]{gur2022convergent}. 
 
 For the $\theta$ subproblem in \eqref{eq:NN_theta} and the $\xi$ subproblem in \eqref{eq:NN_xi}, we have respectively,
	\begin{align}
	&\mathcal{L}(\theta^{k+1},\xi^k)\le \mathcal{L}(\theta^{k},\xi^k),\label{app1.1}\\
	&0\in \partial_{\xi} \mathcal{L}(\theta^{k+1},\xi^k).\label{app1.2}
\end{align}
	Since $\mathcal{L}(\theta,\xi)$ is strongly convex with respect to $\xi$, we can get that for all  $g^{k+1}_\xi\in\partial_{\xi} \mathcal{L}(\theta^{k+1},\xi^{k+1})$, it holds  that
	\begin{align*}
		\mathcal{L}(\theta^{k+1},\xi^{k+1})+\langle g^{k+1}_{\xi},\xi^k-\xi^{k+1}\rangle&+\frac{\mu}{2}\|\xi^k-\xi^{k+1}\|_{2}^2\\
		&\le \mathcal{L}(\theta^{k+1},\xi^k).
	\end{align*}
	Furthermore, adding \eqref{app1.1} to the above inequality and combining it with \eqref{app1.2} yield
	\begin{align}\label{app1.3}
		\frac{\mu}{2}\|\xi^k-\xi^{k+1}\|_{2}^2\le \mathcal{L}(\theta^k,\xi^k)-\mathcal{L}(\theta^{k+1},\xi^{k+1}).
	\end{align}
	It is clear that condition C1 in \cite[Definition 2]{gur2022convergent} holds.
	
	 Secondly, from \eqref{app1.2}, \eqref{equ:subgrad_sum}, and Assumption \ref{ass:bound}, 
  %from the optimality condition of $\theta$ subproblem in \eqref{eq:NN_xi}, and $\mathcal{L}$ is regular at each point $\theta^{k+1}$,
  we can obtain 
	\begin{align*}
		0&\in\partial_{\theta}\mathcal{L}(\theta^{k+1},\xi^k)\\
  &=\partial L(\theta^{k+1}) +\mu\partial f_{\theta^{k+1}}(\IX)^{\top}(f_{\theta^{k+1}}(\IX)-\xi^k)\\
		&=\partial L(\theta^{k+1}) +\mu\partial f_{\theta^{k+1}}(\IX)^{\top}(f_{\theta^{k+1}}(\IX)-\xi^{k+1})\\
		&\quad+\mu\partial f_{\theta^{k+1}}(\IX)^{\top}(\xi^{k+1}-\xi^{k}).
	\end{align*}
	% which means that
	% \begin{equation}\label{app2.1}
	% 	\mu(\xi^{k}-\xi^{k+1})\partial f_{\theta^{k+1}}(\IX)\in\partial_{\theta}\mathcal{L}(\theta^{k+1},\xi^{k+1}).  
	% \end{equation}
	% Combining \eqref{app2.1} with 
     Combing it with \eqref{app1.2}, we get
	\begin{equation*}
		W^{k+1}= \begin{pmatrix}
			\mu g_f^{\top}(\xi^{k}-\xi^{k+1})\\ 
			0
		\end{pmatrix}\in\partial \mathcal{L}(\theta^{k+1},\xi^{k+1}),
	\end{equation*}
	where $g_f\in\partial f_{\theta^{k+1}}(\IX)$.
	Then, with triangle inequality and Assumption \ref{ass:bound}, 
	\begin{equation}\label{app2.2}
		\|W^{k+1}\|\le\mu\|g_f\| \|\xi^{k}-\xi^{k+1}\|\le \mu\rho \|\xi^{k}-\xi^{k+1}\|.
	\end{equation}
	
	At last, if $(\bar\theta,\bar\xi)$ is a limit point of some sub-sequence $\{(\theta^k,\xi^k)\}_{k\in\mathcal{K}\subseteq\mathbb{K}}$, based on 
	the continuity of objective function $L(\theta,\xi)$, we can obtain
	\begin{equation}\label{app3.1}
		\limsup_{k\in\mathcal{K}\subseteq\mathbb{K}} \mathcal{L}(\theta^k,\xi^k)\le \mathcal{L}(\bar\theta,\bar\xi).
	\end{equation}
	Moreover, the condition C4 in \cite[Definition 2]{gur2022convergent} holds clearly when the subproblems \eqref{eq:NN_theta} and \eqref{eq:NN_xi} are solved exactly. Combining \eqref{app1.3}, \eqref{app2.2}, and \eqref{app3.1}, we have   that the sequence $\{(\theta^k,\xi^k)\}_{k\ge0}$ generated by Algorithm \ref{alg:DNN} is an approximate gradient-like descent sequence in \cite{gur2022convergent}. 
 
 Furthermore, let $u:=(\theta,\xi)$ for convenience. The function $\mathcal{L}(u)$  is proper, lower semicontinuous, which has the K{\L} property directly.  
 %based on the fact that a proper lower semicontinuous function  has the K{\L} property  at any noncritical point. 
 Hence, we can get the convergence     with the K{\L} property \cite{bolte2007lojasiewicz} and the approximate gradient-like descent sequence by Theorem 1 in \cite{gur2022convergent}.
	 % there exists $\eta\in(0,+\infty]$ a neighborhood $U$ of $u$ and a function $\psi\in\Psi(\eta)$ such that for all
%	\begin{align}
%		u\in U \cap [\mathcal{L}(\bar{u})\le \mathcal{L}(u)\le \mathcal{L}(\bar{u})+\eta],
%	\end{align}
%	the following inequality holds
%	\begin{align}
%		\psi'(\mathcal{L}(u)-\mathcal{L}(\bar{u}))\hbox{dist}(0,\partial \mathcal{L}(u))\ge 1.
%	\end{align}
%	These mean that $\mathcal{L}$ defined by \eqref{eq:NN_L3} satisfies K{\L} property.
\end{proof}
\begin{remark}
	Actually, the subproblems \eqref{eq:NN_theta} and \eqref{eq:NN_xi} can also be solved inexactly. We can find approximate solution for \eqref{eq:NN_theta} and \eqref{eq:NN_xi} until the following criteria satisfied,
	\begin{align*}
		\frac{\mu}{2}\|\xi^k-\xi^{k+1}\|_{2}^2-(e_1^k)^2&\le \mathcal{L}(\theta^k,\xi^k)-\mathcal{L}(\theta^{k+1},\xi^{k+1}), \\
		\|W^{k+1}\|-e_2^k &\le \mu\rho \|\xi^{k}-\xi^{k+1}\|,
	\end{align*}
	where $\{e_1^k\}_{k\ge0}$ and $\{e_2^k\}_{k\ge0}$ are required to be summable. %Then 
	Together with \eqref{app3.1}, the sequence $\{(\theta^k,\xi^k)\}_{k\ge0}$ is also an approximate gradient-like descend sequence defined in \cite{gur2022convergent}.
\end{remark}

\begin{remark}
	By \cite{jiang2022optimality}, the assumption that $\mathcal{L}$ is regular near  $(\theta^*,\xi^*)$ can be satisfied for some deep neural networks with the ReLU activation function. 
\end{remark}

\section{Numerical Experiments}
In this  section, we verify the efficiency  of our proposed coupled tensor norm regularization for both MLR and DNN on nine real datasets listed in Table \ref{tab:datasets}. 
We first test the performance of MLR on three face image datasets (ORL, Yale, AR10P) and three biological datasets (lung, TOX-171, lymphoma)  downloaded online  \footnote{\url{https://jundongl.github.io/scikit-feature/datasets.html}}. 
%All six datasets have much more features than samples and are prone to overfitting. 
Then we test the performance of DNN on  Fashion-MNIST, CIFAR-10, and an MRI dataset (Brain Tumor) \footnote{\url{https://www.kaggle.com/competitions/machinelearninghackathon/data}}.  

\begin{table}[t!]
	\setlength{\abovecaptionskip}{0.1cm} 
	\caption{Details of all datasets. $n$ and $n'$ are the number of training and  testing samples, respectively.  $m$ denotes the number of features and  $c$ is the number of classes.}
	\centering
	\begin{tabular}{c|rrrr}
		\hline 
		Dataset & $n$ & $n'$ & $m$& $c$\\
		\hline
		ORL  & 280  &120 & 1024 &40 \\
		Yale  & 100  & 65 &1024 &15  \\
		AR10p  & 90 & 40  & 2400&10  \\
		Lung  &  153 & 50 &3312 &5   \\
		TOX-171  & 100  & 71 &  5748 &4 \\
		Lymphoma  & 56   & 40 &4026&9   \\ 
		Fashion-MNIST  & 60000 & 10000 &784&10   \\
		CIFAR-10  & 60000 &10000  & 3072&10  \\
		Brain Tumor  & 2870 &394 &50176  &4   \\
		\hline
	\end{tabular}
	\label{tab:datasets}
	\vspace{-10pt}
\end{table}

%All experiments are implemented through MATLAB(2020a).  
\subsection{Multinomial logistic regression}

\begin{table*}[h!]
	\setlength{\abovecaptionskip}{-0.2cm} 
	\setlength{\belowcaptionskip}{-0.2cm}
	\caption{Numerical results of multinomial logistic regression for three face datasets.}
	\begin{center}
		% \setlength\tabcolsep{4pt}
		%\resizebox{0.99\columnwidth}{!}{
			\begin{tabular}{c|ccc|ccc|ccc}
				\hline
				&   \multicolumn{3}{c|}{ORL} & \multicolumn{3}{c|}{Yale} & \multicolumn{3}{c}{AR10P}\\ \hline
				Model & Training & Testing & $\lambda$ & Training & Testing & $\lambda$  & Training & Testing & $\lambda$ \\ \hline
				
				MLR & 95.71\% &90.83\%  & 0 & 90.00\% & 75.38\% &0 & 85.56\% & 90.00\% & 0\\ 
				MLR-$\ell_1$ & 96.07\% & 93.33\% & $10^{-4}$ & 89.00\% & 75.38\% & $10^{-6}$& 85.56\% &92.50\% & $10^{-3}$\\ 
				MLR-$\ell_2$ &96.07\%  & 93.33\% &$10^{-6}$  & 89.00\% & 75.38\% & $10^{-6}$& 85.56\% & 95.00\% & $10^{-4}$\\ 
				MLR-Tik & 96.42\% & 93.33\% & 1 & 90.00\% & 78.46\% & 0.1 &85.56\% & 97.50 \% & $10^{-2}$\\ 
				MLR-ours & 96.07\% & \textbf{95.00\%} & $10^{-4}$ & 92.00\% & \textbf{81.54\%}  &$10^{-5}$ &85.56\% & \textbf{100\%}  &$10^{-5}$ \\ 
				\hline
			\end{tabular}
			%}
	\end{center}
	\label{tab:face}
\end{table*}

\begin{table*}[h!]
	\setlength{\abovecaptionskip}{-0.2cm} 
	\setlength{\belowcaptionskip}{-0.4cm}
	\caption{Numerical results of multinomial logistic regression for three biological datasets.}
	\begin{center}
		% \setlength\tabcolsep{4pt}
		%\resizebox{0.99\columnwidth}{!}{
			\begin{tabular}{c|ccc|ccc|ccc}
				\hline
				&   \multicolumn{3}{c|}{Lung} & \multicolumn{3}{c|}{TOX-171} & \multicolumn{3}{c}{Lymphoma}\\ \hline
				Model & Training & Testing & $\lambda$ & Training & Testing & $\lambda$  & Training & Testing & $\lambda$ \\ \hline
				
				MLR & 95.43\% & 86.00\%  & 0 & 75.00\% & 57.75\% & 0 & 98.21\% & 85.00\% & 0\\ 
				MLR-$\ell_1$ & 93.46\% & 92.00\% & $10^{-2}$ & 66.00\% & 63.38\% &$10^{-2}$ & 98.21\% & 90.00\%& $10^{-2}$\\ 
				MLR-$\ell_2$ &94.12\%  & 94.00\% &$10^{-2}$  &   66.00\% &61.97\% &$10^{-4}$& 98.21\%  &87.50\% &$10^{-3}$\\ 
				MLR-Tik & 96.08\% & \textbf{96.00\%} & $10^{-1}$& 72.00\% & 67.61\% & 1 & 98.21\% & \textbf{95.00\%} &1 \\ 
				MLR-ours & 96.08\% & \textbf{96.00\%} & $10^{-2}$ & 71.00\% & \textbf{69.01\%} & $10^{-2}$ & 98.21\% &\textbf{95.00\%} &1 \\ 
				\hline
			\end{tabular}
			%}
	\end{center}
	\label{tab:biological}
\end{table*}

\begin{table*}[h!]
	\setlength{\abovecaptionskip}{-0.2cm} 
	\setlength{\belowcaptionskip}{-0.4cm}
	\caption{Comparisons of difference $\lambda$ between ours and Tikhonov regularization for MLR on Lymphoma dataset.}
	\begin{center}
		\begin{tabular}{c|ccccccc}
			%\hline
			% & \multicolumn{7} $\lambda$\\
			\hline
			$\lambda$& $1$ &$10^{-1}$ &$10^{-2}$ &$10^{-3}$ & $10^{-4}$ & $10^{-5}$ & $10^{-6}$\\
			\hline
			MLR-Tik & \textbf{95.00\%} &82.50\% &67.50\% &42.50\% &42.50\% &45.00\% &45.00\%\\
			MLR-ours & \textbf{95.00\%} & 92.50\%& 87.50\% & 82.50\% &82.50\% &82.50\% &82.50\% \\
			\hline
		\end{tabular}
	\end{center}
	\label{tab:robust1}
	\vspace{-13pt}
\end{table*}

%Training samples and testing samples are chosen randomly with certain amounts. It's worth mentioning that training samples are part of testing samples in AR10p, lymphoma datasets owing to the size of data is small. 

%In order to verify the validity of coupled tensor norm, 
In this subsection, we compare the coupled tensor norm regularization model \eqref{MLR_R} with the $\ell^1$-norm \cite{schmidt2007fast}, $\ell^2$-norm \cite{ndiaye2015gap}, Tikhonov regularization \cite{bishop1995training} models. 
%for each weight vector $w_k,\;\forall\; 1\le k\le c$. 
For all regularized models, we traverse   $\lambda$ from $\{10^{-6},10^{-5},10^{-4},10^{-3},\\10^{-2},10^{-1},1\}$ and report the results corresponding to $\lambda$ with the highest classification accuracy. 
The gradient or subgradient descent algorithm is adopted to solve the models. 
%with different regularizations for equity. 
%\subsubsection{Experimental setups}
%All experiments adopt the same setups for fair comparisons. 
%equity.
%To choose the parameter $\lambda$ of multinomial logistic regression in \eqref{MLR_R}, 
The corresponding stopping criteria is set as  
\begin{align*}
	\|W^{k+1}-W^k\|_{F}\le 10^{-4} \;\;\hbox{or}\;\; \|\nabla G(W^k)\|_{2}\le 10^{-4},
\end{align*}
and the maximum number of iterations is 2000. Moreover, we initialize $W^0=0$.

The    training accuracy, testing accuracy, and the choices of the optimal parameters 
on the face and biological datasets are elaborated in Tables \ref{tab:face} and   \ref{tab:biological}, respectively. 
%As presented in Table \ref{tab:face}, our regularization guarantees the highest testing accuracy  and lowest generalization error for all three face datasets.  Table \ref{tab:biological} shows that both our regularization and Tikhonov regularization have better testing accuracy than the baselines.  
For all six datasets, our regularization guarantees the highest testing accuracy  and lowest generalization error. 
%The testing accuracy  is also summarized in Figure \ref{fig:results}. 
We further compare the coupled tensor norm  and Tikhonov regularizations for all   $\lambda$    in  Table~\ref{tab:robust1}, which shows that our  regularization is more  robust.

\subsection{Deep neural networks}
We continue to compare the performance of DNN with the coupled norm regularization with the $\ell_1$ norm \cite{ma2019transformed} and Tikhonov regularization \cite{krogh1991simple}. 
The network structures we tested %include Alexnet \cite{krizhevsky2012imagenet} and 
VGG-16 \cite{simonyan2014very}. 
Also, we verify the efficiency of our proposed method by setting the number of training samples from small to large. 

For all methods, the hyperparameters $\lambda$ and $\mu$ are optimized from 
%$\{10^{-6},5\cdot 10^{-6},10^{-5},5\cdot 10^{-5},10^{-4},5\cdot 10^{-4},10^{-3},5\cdot 10^{-3},10^{-2},5\cdot 10^{-2},10^{-1},5\cdot 10^{-1}\}$ 
$\{10^{-i},5\cdot 10^{-i}\}_{i=1}^6$
and we only report the best performance. The implementation details and the choices of hyperparameters are given in the appendix.
For Fashion-MNIST, we show the performance of different regularizers   with varying training sizes from 1000 to 60000. 
%The initial learning rate $r_0=0.01$ and 
The detailed result is shown in Table \ref{tab:MNI_result1}. %\ref{tab:MNI_result2}. 
%the hyperparameters are reported in Table \ref{tab:MNI_para}. 
% the test accuracy of different regularization networks on different sizes of Fashion-MNIST datasets. 
%We also show  the relationship of testing error rate with the size of training set   in Figure \ref{fig:error}. 
%It is clear to see that our coupled norm regularizer outperforms $\ell_1$ and Tikhonov, 
%especially when the training set is small. 
%\begin{figure}[!ht]
%	\centering
%	\includegraphics[width=0.23\textwidth]{fig/MNI_Ale_err.pdf} 
%	\includegraphics[width=0.23\textwidth]{fig/MNI_VGG_err.pdf}
	% 	\\
	% 	\includegraphics[width=0.23\textwidth]{fig/}
	% 	\includegraphics[width=0.23\textwidth]{fig/}
%	\caption{Comparison of different regularizers on the Fashion-MNIST dataset. The left and right figures are the results for Alexnet and   VGG-16, respectively. 
%	}
%	\label{fig:error}
%\end{figure}

\begin{table}[!h]
	\setlength{\abovecaptionskip}{-0.2cm} 
	\setlength{\belowcaptionskip}{0.1cm}
	\caption{The testing accuracy of different regularizers for VGG-16 on Fashion-MNIST.}
	\begin{center}
		\begin{tabular}{c|cccc}
			\hline
			Model & \multicolumn{4}{c}{VGG-16 on Fashion-MNIST}             \\\hline
			\begin{tabular}[c]{@{}c@{}}Training \\ per class\end{tabular} & DNN      & DNN-$\ell_1$ & DNN-Tik & DNN-ours \\\hline
			100 & 80.95\% & 82.40\%  & 82.16\%   & \textbf{83.38\%}  \\
			400                                                           & 86.95\% & 87.78\%  & 87.13\%    & \textbf{88.15\%}  \\
			
			700                                                           & 88.60\% & 90.03\%    & 89.66\%    & \textbf{90.73\%} \\
			1000                                                           & 90.67\% & 90.85\%    & 90.76\%    & \textbf{91.28\%} \\
			3000                                                           & 92.13\% & 92.70\%    & 92.62\%    & \textbf{92.93\%} \\
			6000                                                           & 93.88\% & 94.29\%    & 94.30\%    & \textbf{94.73\%} \\\hline
		\end{tabular}
	\end{center}
	\label{tab:MNI_result1}
	\vspace{-22pt}
\end{table}

At last, we present the results for CIFAR-10 and  Brain Tumor  in Tables  \ref{tab:CIF_result} and \ref{tab:MRI_result}, respectively. 
The numerical experiments show that the generalization ability of our coupled norm regularizer is better than the baselines.

%\begin{table}[h!]
%	\setlength{\abovecaptionskip}{-0.2cm} 
%	\setlength{\belowcaptionskip}{-0.2cm}
%	\caption{The testing accuracy of different regularizers for Alexnet on Fashion-MNIST.}
%	\begin{center}
%		\begin{tabular}{c|cccc}
%			\hline
%			Model & \multicolumn{4}{c}{Alexnet on Fashion-MNIST}             \\\hline
%			\begin{tabular}[c]{@{}c@{}}Training \\ per class\end{tabular} & DNN      & DNN-$\ell_1$ & DNN-Tik & DNN-ours \\\hline
%			100 & 75.55\% & 76.89\%  & 77.09\%   & \textbf{79.11\%}  \\
%			400                                                           & 84.39\% & 85.39\%  & 84.54\%    & \textbf{86.63\%}  \\
%			
%			700                                                           & 86.36\% & 86.85\%    & 86.93\%    & \textbf{88.01\%} \\
%			1000                                                           & 87.61\% & 87.91\%    & 87.86\%    & \textbf{89.01\%} \\
%			3000                                                           & 88.87\% & 89.32\%    & 89.48\%    & \textbf{89.64\%} \\
%			6000                                                           & 90.31\% & 90.93\%    & 90.74\%    & \textbf{91.12\%} \\\hline
%		\end{tabular}
%	\end{center}
%	\label{tab:MNI_result2}
%\end{table}

\begin{table}[h!]
	\setlength{\abovecaptionskip}{-0.2cm} 
	\setlength{\belowcaptionskip}{-0.1cm}
	\caption{The testing accuracy of different regularizers for VGG-16 on CIFAR-10.}
	\begin{center}
	\begin{tabular}{c|cccc}
		\hline
		Model                                                         & \multicolumn{4}{c}{VGG-16 on CIFAR-10}             \\\hline
		\begin{tabular}[c]{@{}c@{}}Training \\ per class\end{tabular} & DNN      & DNN-$\ell_1$ & DNN-Tik & DNN-ours \\\hline
		100                                                           & 48.57\% & 49.10\%  & 48.98\%   & \textbf{50.26\%}  \\
		400                                                           & 72.90\% & 73.29\%  & 73.32\%    & \textbf{74.15\%}  \\
		700                                                           & 78.97\% & 79.14\%    & 79.31\%    & \textbf{80.29\%} \\\hline
	\end{tabular}
    \end{center}
	\label{tab:CIF_result}
 \vspace{-20pt}
\end{table}

\begin{table}[h!]
	\setlength{\abovecaptionskip}{-0.2cm} 
	\setlength{\belowcaptionskip}{0cm}
	\caption{Numerical results of different regularizations for VGG-16 on the MRI dataset Brain Tumor.}
	\begin{center}
	\begin{tabular}{c|ccc}
		\hline
		& \multicolumn{3}{c}{VGG-16 on Brain Tumor}                                                            \\\hline
		Model                        & Training                    & Testing                     & $\lambda~(\mu)$                     \\\hline
		DNN                          & 99.50\%                     & 75.48\%                     & 0                                  \\
		DNN-$\ell_1$                 & 99.97\%                     & 76.40\%                     & $10^{-3}$                              \\
		DNN-Tik                      & 99.83\%                     & 76.67\%                     & $10^{-3}$                              \\
		DNN-ours     & 99.97\%  & \textbf{77.41\%} & $5\cdot 10^{-4}$ ($10^{-4}$)\\\hline
	\end{tabular}
    \end{center}
	\label{tab:MRI_result}
	\vspace{-20pt}
\end{table}

\section{Conclusions}
In this paper, we proposed a coupled tensor norm approach to obtain better
generalization for classification models. 
%Our %coupled tensor norm  
%regularization is data-dependent, which mainly enables the input data and the output feature to lie in a low-dimensional manifold. 
Theoretically, for MLR, we showed this regularization is convex, differentiable, and gradient Lipschitz continuous and proved the global convergence of the gradient descent method. For DNN, we showed the regularization is nonconvex and nonsmooth, and established the global convergence of the alternating minimization method. %based on the K{\L} property. 
At last, we verified the efficiency of our  regularization compared with the $\ell_1$, $\ell_2$, and Tikhonov regularizations.

\section*{Acknowledgement}

This research is supported by the National Natural Science Foundation of China (NSFC) grants 12131004, 12126603, 
12126608, KZ37099001, and
KZ77010604.

\onecolumn
\section{Appendix}

Details of numerical implementation for DNN: Unless otherwise stated, all experiments use    SGD with momentum fixed at 0.9 and mini-batch size fixed as 128.  The networks are trained with a fixed learning rate $r_0=0.01$ on the first 50 epochs, and then $r_0/10$ for another 50 epochs. At step 1 of Algorithm \ref{alg:DNN}, $\theta$ is updated once every $M=2$ epochs of SGD. And at step 2, the step size is set to $1/k$. The stopping criterion is $\|grad_\xi\|_F < tol$, where $grad_\xi$ is the subgradient of \eqref{eq:NN_xi}. Further, we set $tol=10^{-2}$ and the maximum number of iterations as 50. 
%\begin{table}[!ht]
%	\caption{Hyperparameters of DNN for Fashion-MNIST dataset.}
%	\begin{center}
%		\begin{tabular}{c|cccc}
%			\hline
%			Model & \multicolumn{4}{c}{Alexnet}                       \\\hline
%			\multirow{2}{*}{\begin{tabular}[c]{@{}c@{}}Training \\ per class\end{tabular}} 
%			& DNN-$\ell_1$ & DNN-Tik    & \multicolumn{2}{c}{ DNN-ours} \\
%			& $\lambda$ & $\lambda$ & $\lambda$ &       $\mu$  \\\hline
%			100                                                                                & $10^{-4}$    & $10^{-5}$   & $5\cdot 10^{-4}$        & $5\cdot 10^{-2}$     \\
%			400                                                                                    & $10^{-4}$    & $10^{-4}$    & $5\cdot 10^{-4}$         & $10^{-2}$     \\
%			700                                                                                   & $10^{-4}$     & $10^{-5}$   & $5\cdot 10^{-4}$         & $5\cdot 10^{-2}$              \\
%			1000                                                                                  & $10^{-4}$    & $10^{-5}$   & $5\cdot 10^{-4}$         & $5\cdot 10^{-3}$      \\
%			3000                                                                                & $10^{-4}$    & $10^{-4}$    & $5\cdot 10^{-4}$         & $5\cdot 10^{-3}$    \\
%			6000                                                                                 & $5\cdot 10^{-3}$     & $10^{-4}$   & $5\cdot 10^{-4}$         & $5\cdot 10^{-3}$      \\\hline        
%		\end{tabular}
%	\end{center}
%	\label{tab:MNI_para1}
%\end{table}

\begin{table}[!ht]
	\caption{Hyperparameters of DNN for Fashion-MNIST dataset.}
	\begin{center}
		\begin{tabular}{c|cccc}
			\hline
			Model & \multicolumn{4}{c}{VGG-16}                       \\\hline
			\multirow{2}{*}{\begin{tabular}[c]{@{}c@{}}Training \\ per class\end{tabular}} 
			& DNN-$\ell_1$ & DNN-Tik    & \multicolumn{2}{c}{ DNN-ours} \\
			& $\lambda$ & $\lambda$ & $\lambda$ &       $\mu$  \\\hline
			100                                                                                & $10^{-5}$    & $10^{-5}$   & $5\cdot 10^{-4}$        & $5\cdot 10^{-3}$     \\
			400                                                                                    & $10^{-5}$    & $10^{-5}$    & $5\cdot 10^{-4}$         & $10^{-3}$     \\
			700                                                                                   & $10^{-5}$     & $10^{-5}$   & $5\cdot 10^{-4}$         & $5\cdot 10^{-4}$              \\
			1000                                                                                  & $10^{-4}$    & $10^{-5}$   & $5\cdot 10^{-4}$         & $5\cdot 10^{-4}$      \\
			3000                                                                                & $10^{-4}$    & $10^{-5}$    & $5\cdot 10^{-4}$         & $5\cdot 10^{-4}$    \\
			6000                                                                                 & $ 10^{-4}$     & $10^{-5}$   & $5\cdot 10^{-4}$         & $5\cdot 10^{-4}$      \\\hline        
		\end{tabular}
	\end{center}
	\label{tab:MNI_para2}
\end{table}

\begin{table}[!ht]
	\caption{Hyperparameters used in CIFAR-10 dataset.}
	\begin{center}
		\begin{tabular}{c|ccccc}
			\hline
			Model                                                                          & \multicolumn{5}{c}{VGG-16}                                     \\\hline
			\multirow{2}{*}{\begin{tabular}[c]{@{}c@{}}Training \\ per class\end{tabular}} & DNN        & DNN-$\ell_1$  & DNN-Tik  & \multicolumn{2}{c}{DNN-ours} \\
			& $\lambda$ & $\lambda$ & $\lambda$ & $\lambda$      & $\mu$     \\\hline
			100                                                                            & 0         & $10^{-4}$    & $10^{-5}$   & $5\cdot 10^{-5}$        & $5\cdot 10^{-2}$      \\
			400                                                                            & 0         & $10^{-4}$    & $10^{-3}$     & $5\cdot 10^{-4}$         & $5\cdot 10^{-2}$      \\
			700                                                                            & 0         & $10^{-3}$     & $10^{-5}$   & $5\cdot 10^{-4}$         & $5\cdot 10^{-2}$  \\\hline   
		\end{tabular}
	\end{center}
	\label{tab:CIF_para}
\end{table}


\begin{thebibliography}{00}
	
	
\bibitem{rockafellar2009variational}
 R.T.~Rockafellar, R.J-B~Wets, Variational analysis, Springer Science \& Business Media, 2009.



\bibitem{AcKoDu11}
E.~Acar, T.G. Kolda, D.M. Dunlavy,
All-at-once
optimization for coupled matrix and tensor factorizations, in: MLG'11:
Proceedings of Mining and Learning with Graphs, 2011.

\bibitem{arora2018stronger}
S.~Arora, R.~Ge, B.~Neyshabur, Y.~Zhang, Stronger generalization bounds for
deep nets via a compression approach, in: International Conference on Machine
Learning, PMLR, 2018, pp. 254-263.

\bibitem{bishop1995training}
C.M. Bishop, Training with noise is equivalent to Tikhonov regularization,
Neural comput. 7~(1) (1995) 108--116.

\bibitem{bolte2007lojasiewicz}
J.~Bolte, A.~Daniilidis, A.~Lewis, The {\L}ojasiewicz inequality for nonsmooth
subanalytic functions with applications to subgradient dynamical systems,
SIAM J. Optim. 17~(4) (2007) 1205--1223.

\bibitem{cohen2016convolutional}
N.~Cohen, A.~Shashua, Convolutional rectifier networks as generalized tensor
decompositions, in: International Conference on Machine Learning, PMLR, 2016,
pp. 955--963.

\bibitem{cortes2012l2}
C.~Cortes, M.~Mohri, A.~Rostamizadeh, L2 regularization for learning kernels,
in: Proceedings of the 25th Conference on Uncertainty in Artificial
Intelligence, 2009, pp. 109--116.

%\bibitem{dharmadasa2018radio}
%M.~Dharmadasa, C.~Gamage, C.~Keppitiyagama, Radio tomographic imaging (RTI) and
%privacy implications, in: 2018 18th International Conference on Advances in
%ICT for Emerging Regions (ICTer), IEEE, 2018, pp. 413--419.

\bibitem{gur2022convergent}
E.~Gur, S.~Sabach, S.~Shtern, Convergent nested alternating minimization
algorithms for nonconvex optimization problems, Math. Oper. Res. (2022).

\bibitem{ioffe2015batch}
S.~Ioffe, C.~Szegedy, Batch normalization: Accelerating deep network training
by reducing internal covariate shift, in: International Conference on Machine
Learning, PMLR, 2015, pp. 448--456.

\bibitem{kolbeinsson2021tensor}
A.~Kolbeinsson, J.~Kossaifi, Y.~Panagakis, A.~Bulat, A.~Anandkumar,
I.~Tzoulaki, P.~M. Matthews, Tensor dropout for robust learning, IEEE J. Sel. Topics Signal Process. 15~(3) (2021) 630--640.

\bibitem{kolda2009}
T.G. Kolda, B.W. Bader, Tensor decompositions and applications, SIAM Rev.
51~(3) (2009) 455--500.
%\newblock \href {https://doi.org/10.1137/07070111X}
%{\path{doi:10.1137/07070111X}}.

\bibitem{krizhevsky2012imagenet}
A.~Krizhevsky, I.~Sutskever, G.~E. Hinton, Imagenet classification with deep
convolutional neural networks, in: Advances in Neural Information Processing
Systems, 2012, pp. 1097-1105.

\bibitem{krogh1991simple}
A.~Krogh, J.~Hertz, A simple weight decay can improve generalization, in: Advances
 Neural Information Processing Systems, 1991, pp. 950-957.

\bibitem{Li2020understanding}
J.~Li, Y.~Sun, J.~Su, T.~Suzuki, F.~Huang, Understanding generalization in deep
learning via tensor methods, in: The 23rd International Conference on Artificial Intelligence and Statistics, PMLR, 2020, pp. 504--515.

\bibitem{ma2019transformed}
R.~Ma, J.~Miao, L.~Niu, P.~Zhang, Transformed $\ell_1$ regularization for
learning sparse deep neural networks, Neural Netw. 119 (2019) 286--298.

\bibitem{lyu2022improving}
S.-H. Lyu, L.~Wang, Z.-H. Zhou, Improving generalization of deep neural
networks by leveraging margin distribution, Neural Netw. 151 (2022)
48--60.

\bibitem{more1994line}
J.J. Mor{\'e}, D.J. Thuente, Line search algorithms with guaranteed
sufficient decrease, ACM Trans Math Softw. 20~(3)
(1994) 286--307.

\bibitem{ndiaye2015gap}
E.~Ndiaye, O.~Fercoq, A.~Gramfort, J.~Salmon, Gap safe screening rules for
sparse multi-task and multi-class models, in: Advances in Neural Information
Processing Systems, 2015, pp. 811-819 .

\bibitem{osher2017low}
S.~Osher, Z.~Shi, W.~Zhu, Low dimensional manifold model for image processing,
SIAM J Imaging Sci. 10~(4) (2017) 1669--1690.

\bibitem{panagakis2021tensor}
Y.~Panagakis, J.~Kossaifi, G.G.~ Chrysos, J.~Oldfield, M.A.~ Nicolaou, A.~Anandkumar, S.~Zafeiriou, Tensor methods in computer vision and deep learning, Proc. IEEE  109~(5) (2021) 863-890.


\bibitem{peyre2008image}
G.~Peyr{\'e}, Image processing with nonlocal spectral bases, 	Multiscale Model. Simul. 7~(2) (2008) pp. 703--730.

\bibitem{peyre2009manifold}
G.~Peyr{\'e}, Manifold models for signals and images, Comput Vis Image Underst. 113~(2) (2009) 249--260.

% \bibitem{sanyal2019stable}
% A.~Sanyal, P.H.~ Torr, P.K.~ Dokania,
% Stable rank normalization
% for improved generalization in neural networks and {GAN}s, in: International
% Conference on Learning Representations, 2020.

\bibitem{schmidt2007fast}
M.~Schmidt, G.~Fung, R.~Rosales, Fast optimization methods for $\ell_1$
regularization: A comparative study and two new approaches, in: European
Conference on Machine Learning, Springer, 2007, pp. 286--297.

\bibitem{shekar2020l1}
B.~Shekar, G.~Dagnew, L1-regulated feature selection and classification of
microarray cancer data using deep learning, in: Proceedings of 3rd
international conference on computer vision and image processing, Springer,
2020, pp. 227--242.


\bibitem{signoretto2010nuclear}
M.~Signoretto, L.~{De Lathauwer}, J.A.~ Suykens, Nuclear norms for tensors and
their use for convex multilinear estimation (2010).

\bibitem{srivastava2014dropout}
N.~Srivastava, G.~Hinton, A.~Krizhevsky, I.~Sutskever, R.~Salakhutdinov,
Dropout: a simple way to prevent neural networks from overfitting, J Mach Learn Res. 15~(1) (2014) 1929--1958.

\bibitem{sun2006optimization}
W.~Sun, Y.-X.~ Yuan, Optimization theory and methods: nonlinear programming,
vol.~1, Springer Science \& Business Media, 2006.

\bibitem{tomioka2013convex}
R.~Tomioka, T.~Suzuki, Convex tensor decomposition via structured schatten norm
regularization, in: Advances in Neural Information Processing Systems, 2013, pp. 1331–1339.

\bibitem{vapnik1999overview}
V.N.~ Vapnik, An overview of statistical learning theory, IEEE Trans. Neural Netw. 10~(5) (1999) 988--999.

\bibitem{wimalawarne2018efficient}
K.~Wimalawarne, H.~Mamitsuka, Efficient convex completion of coupled tensors
using coupled nuclear norms, in: Advances in Neural Information Processing Systems, 2018, pp. 6902-6910.

\bibitem{wimalawarne2014multitask}
K.~Wimalawarne, M.~Sugiyama, R.~Tomioka, Multitask learning meets tensor
factorization: task imputation via convex optimization, in: Advances in Neural
Information Processing Systems, 2014, pp. 2825–2833.

\bibitem{wimalawarne2018convex}
K.~Wimalawarne, M.~Yamada, H.~Mamitsuka, Convex coupled matrix and tensor
completion, Neural Comput. 30~(11) (2018) 3095--3127.

\bibitem{ZHANG2022148}
P.~Zhang, R.~Wang, N.~Xiu, Multinomial logistic regression classifier via
$\ell_{q,0}$-proximal newton algorithm, Neurocomputing 468 (2022) 148--164.

\bibitem{LDMNet}
W.~Zhu, Q.~Qiu, J.~Huang, R.~Calderbank, G.~Sapiro, I.~Daubechies, LDMnet: Low
dimensional manifold regularized neural networks, in: 2018 IEEE/CVF
Conference on Computer Vision and Pattern Recognition, 2018, pp. 2743--2751.

\bibitem{simonyan2014very}
K.~Simonyan, A.~Zisserman, Very deep convolutional networks for large-scale
image recognition, arXiv preprint arXiv:1409.1556 (2014).

\bibitem{jiang2022optimality}
J.~ Jiang, X.~ Chen, Optimality conditions for nonsmooth nonconvex-nonconcave min-max problems and generative adversarial networks, arXiv preprint arXiv:2203.10914(2022).












%\bibitem{he2016deep}
%K.~He, X.~Zhang, S.~Ren, J.~Sun, Deep residual learning for image recognition,
%in: Proceedings of the IEEE conference on computer vision and pattern
%recognition, 2016, pp. 770--778.

%%\bibitem{rabusseau2016low}
%G.~Rabusseau, H.~Kadri, Low-rank regression with tensor responses, Advances in
%Neural Information Processing Systems 29 (2016).
%
%\bibitem{guo2011tensor}
%W.~Guo, I.~Kotsia, I.~Patras, Tensor learning for regression, IEEE Transactions
%on Image Processing 21~(2) (2011) 816--827.
%
%\bibitem{zhou2013tensor}
%H.~Zhou, L.~Li, H.~Zhu, Tensor regression with applications in neuroimaging
%data analysis, Journal of the American Statistical Association 108~(502)
%(2013) 540--552.
%
%\bibitem{romera2013multilinear}
%B.~Romera-Paredes, H.~Aung, N.~Bianchi-Berthouze, M.~Pontil, Multilinear
%multitask learning, in: International Conference on Machine Learning, PMLR,
%2013, pp. 1444--1452.
%
%
%
%
%
%
%
%
%
%\bibitem{loshchilov2017decoupled}
%I.~Loshchilov, F.~Hutter, Decoupled weight decay regularization (2019).
%
%
%
%
%
%
%
%
%
%\bibitem{tomioka2010estimation}
%R.~Tomioka, K.~Hayashi, H.~Kashima, Estimation of low-rank tensors via convex
%optimization, arXiv preprint arXiv:1010.0789 (2010).
%
%
%
%\bibitem{wan2013regularization}
%L.~Wan, M.~Zeiler, S.~Zhang, Y.~Le~Cun, R.~Fergus, Regularization of neural
%networks using dropconnect, in: International conference on machine learning,
%PMLR, 2013, pp. 1058--1066.
%
%\bibitem{nowlan2018simplifying}
%S.~J. Nowlan, G.~E. Hinton, Simplifying neural networks by soft weight sharing,
%in: The Mathematics of Generalization, CRC Press, 2018, pp. 373--394.
%
%\bibitem{zhang2016l1}
%Y.~Zhang, J.~D. Lee, M.~I. Jordan, $\ell_1$-regularized neural networks are
%improperly learnable in polynomial time, in: International Conference on
%Machine Learning, PMLR, 2016, pp. 993--1001.
%
%
%
%\bibitem{burden2008bayesian}
%F.~Burden, D.~Winkler, Bayesian regularization of neural networks, Artificial
%neural networks (2008) 23--42.
%
%
%%\newblock \href {https://doi.org/10.1109/CVPR.2018.00290}
%%{\path{doi:10.1109/CVPR.2018.00290}}.
%
%
%
%
%%\newblock \href {http://arxiv.org/abs/1105.3422} {\path{arXiv:1105.3422}}.
%%\newline\urlprefix\url{https://www.cs.purdue.edu/mlg2011/papers/paper_4.pdf}
%
%
%%\newblock \href {https://doi.org/https://doi.org/10.1016/j.neucom.2021.10.005}
%%{\path{doi:https://doi.org/10.1016/j.neucom.2021.10.005}}.
%
%
%
%
%
%
%
%\bibitem{cong2019ct}
%W.~Cong, G.~Wang, Q.~Yang, J.~Li, J.~Hsieh, R.~Lai, Ct image reconstruction on
%a low dimensional manifold, Inverse Problems \& Imaging 13~(3) (2019) 449.
%
%\bibitem{niu2021low}
%C.~Niu, W.~Cong, F.-L. Fan, H.~Shan, M.~Li, J.~Liang, G.~Wang, Low-dimensional
%manifold constrained disentanglement network for metal artifact reduction,
%IEEE Transactions on Radiation and Plasma Medical Sciences (2021).
%
%
%
%
%
%
%
%\bibitem{wright1999numerical}
%S.~Wright, J.~Nocedal, et~al., Numerical optimization, Springer Science
%35~(67-68) (1999) 7.

%\bibitem{xiao2017fashion}
%H.~Xiao, K.~Rasul, R.~Vollgraf, Fashion-mnist: a novel image dataset for
%benchmarking machine learning algorithms, arXiv preprint arXiv:1708.07747
%(2017).



%\bibitem{bolte2014proximal}
%J.~Bolte, S.~Sabach, M.~Teboulle, Proximal alternating linearized minimization
%for nonconvex and nonsmooth problems, Mathematical Programming 146~(1) (2014)
%459--494.
%
%\bibitem{attouch2010proximal}
%H.~Attouch, J.~Bolte, P.~Redont, A.~Soubeyran, Proximal alternating
%minimization and projection methods for nonconvex problems: An approach based
%on the kurdyka-{\l}ojasiewicz inequality, Mathematics of operations research
%35~(2) (2010) 438--457.


\end{thebibliography}
\end{document}